\documentclass[11pt]{amsart}

\usepackage{amssymb,amsthm,graphicx,amsmath, verbatim, color, mathtools,hyperref,lineno} % mathematical fonts (in amsfonts)
\usepackage[mathscr]{eucal} % load euler fonts (in amsfonts)
\usepackage[matrix,arrow,tips,curve,ps]{xy}
\usepackage{amsmath} %this is needed for \rxar (\xrightarrow)
\usepackage{epsfig}
\usepackage{amscd}
\usepackage[usenames,dvipsnames,svgnames,table]{xcolor}
\usepackage{tikz}
\usepackage{tikz-cd}
\usepackage{mathrsfs}
\usepackage{verbatim}

\newtheorem{TEO}{Theorem}[section]

\newtheorem{LEM}[TEO]{Lemma}
\newtheorem{PLEM}[TEO]{Projection Lemma}

\newtheorem{COR}[TEO]{Corollary}

\newtheorem{CLA}[TEO]{Claim}

\theoremstyle{definition}
\theoremstyle{remark}
\newtheorem{remark}{Remark}
\newtheoremstyle{dico}% name of the style to be used
 {\baselineskip}   % ABOVESPACE
  {\topsep}   % BELOWSPACE
  {}  % BODYFONT
  {0pt}       % INDENT (empty value is the same as 0pt)
  {} % HEADFONT
  {.}         % HEADPUNCT
  {5pt plus 1pt minus 1pt} % HEADSPACE
  {}          % CUSTOM-HEAD-SPEC
\theoremstyle{dico}

\numberwithin{equation}{section}

\newcommand\dual{\mathrel{\raise3pt\hbox{$\underline{\mathrm{\thinspace d
\thinspace}}$}}}

\newcommand\Z{\mathbb Z}

\newcommand\Co{\mathbb C}

\renewcommand{\phi}{\varphi}

\newcommand{\mihi}[1]{}

\begin{document}

%\footnote{
%AMS Subject classification:  }

\title{Degree of irrationality of a very general abelian variety}

\author[E. Colombo]{E. Colombo}
\address{Dipartimento di Matematica,
Universit\`a di Milano, via Saldini 50,
     I-20133, Milano, Italy } \email{{\tt
elisabetta.colombo@unimi.it}}

\author[O. Martin]{O. Martin}
\address{Department of Mathematics, University of Chicago,
5734 S. University Avenue,
Chicago, IL, 60637, USA}  \email{{\tt
olivier@uchicago.edu}}

\author[J.C. Naranjo]{J. C. Naranjo}
\address{Departament de Matem\`atiques i Inform\`atica, Universitat de Barcelona, Gran Via de les Corts Catalanes 585,
08007, Barcelona, Spain  } \email{{\tt
jcnaranjo@ub.edu }}
\thanks{J.C. Naranjo is partially supported by the Proyecto de Investigaci\'on MTM2015-65361-P}

\author[G.P. Pirola]{G. P. Pirola}
\address{Dipartimento di Matematica, Universit\`a di Pavia,
via Ferrata 5, I-27100 Pavia, Italy } \email{{\tt
pietro.pirola@unipv.it}}
\thanks{E. Colombo and G.P. Pirola are members of Gnsaga (INDAM) and are partially supported by PRIN project \emph{Moduli spaces and Lie theory (2017) }, G.P. Pirola is partially supported by MIUR: Dipartimenti di Eccellenza Program (2018-2022) - Dept. of Math. Univ. of Pavia.}
\thanks{O. Martin acknowledges the support of the Natural Sciences and Engineering Research Council of Canada (NSERC). O. Martin est partiellement financ\'e par le Conseil de recherches en sciences naturelles et en g\'enie du Canada (CRSNG)}

%\date{\today}

\setlength{\parskip}{.1 in}

\begin{abstract}
Consider a very general abelian variety $A$ of dimension at least $3$ and an integer $0<d\leq \dim A$. We show that if the map $A^k\to CH_0(A)$ has a $d$-dimensional fiber then $k\geq d+(\dim A+1)/2$. This extends results of the second-named author which covered the cases $d=1,2$. As a geometric application, we obtain that any dominant rational map from a very general abelian $g$-fold to $\mathbb{P}^g$ has degree at least $(3\dim A+1)/2$ for $g\geq 3$. This improves results of Alzati and the last-named author in the case of a very general abelian variety.

\end{abstract}

\maketitle

\section{Introduction and Results}

Given an irreducible $n$-dimensional projective variety $X$, the \emph{degree of irrationality} $\text{irr}(X)$ of $X$ is the minimal degree of a dominant rational map from $X$ to $\mathbb{P}^n$ (\cite{MH}). It is a birational invariant that measures how far $X$ is from being rational and, accordingly, one expects the computation of $\text{irr}(X)$ to be a difficult problem. Little is known about $\text{irr}(X)$ even in the case of abelian varieties. In Section 4 of \cite{So} Sommese showed by topological methods that the minimal degree of a morphism from an abelian $g$-fold to $\mathbb{P}^g$ is at least $g+1$. Using inequalities on the holomorphic length, Alzati and the last-named author extended these results to rational maps in \cite{AP2}. This showed that any abelian $g$-fold has degree of irrationality at least $g+1$.

 A few years later, Tokunaga and Yoshihara (\cite{TY}) proved that if an abelian surface contains a smooth curve of genus 3 then it admits a degree $3$ rational map to $\mathbb{P}^2$. This allowed them to show that the bound $\text{irr}(A)\geq \dim A+1$ is sharp for abelian surfaces by providing an example of an elliptic curve $E$ with complex multiplication such that $\text{irr}(E\times E)=3$. In \cite{Y}, Yoshihara goes on to exhibit a countable collection of such elliptic curves $E$. He also asks whether the degree of irrationality is an isogeny invariant, if $\text{irr}(E_1\times E_2)\geq 4$ for $E_1,E_2$ non-isogenous elliptic curves, and if there are examples of abelian surfaces with degree of irrationality at least $4$. To our knowledge all these questions are open to this day. We also do not know if the bound $\text{irr}(A)\geq \dim A+1$ is sharp for higher-dimensional abelian varieties.
 
More recently, in \cite{BDPELU} Bastianelli, de Poi, Ein, Lazarsfeld, and Ullery conjectured that 
$$\lim_{d\to\infty}\text{irr}(A_d)=\infty,$$
 where $A_d$ is a very general abelian surface with polarization of type $(1,d)$. This was disproved by Chen in \cite{C}, in which he shows that $\text{irr}(A_d)\leq 4$. Whether the degree of irrationality of a very general abelian variety of higher dimension depends on the degree of the polarization remains open. Finally, in \cite{M} the second-named author obtains as a by-product of his method that the degree of irrationality of a very general abelian $g$-fold is at least $g+2$ for $g\geq 3$. This is a slight improvement of the previous bound $\text{irr}(A)\geq \dim A+1$ in the very general case and it shows for the first time the existence of abelian varieties for which it fails to be optimal.

In this short note we show that the induction proposed in \cite{M} can be applied in order to obtain new bounds on the degree of irrationality of very general abelian varieties. As in \cite{P},\cite{AP},\cite{V}, and \cite{M}, we study rational equivalence of effective zero-cycles on a very general abelian variety. The connection between rational equivalence and measures of irrationality is provided by the following.

\begin{remark}\label{remark}
A dominant rational map $f: X\dasharrow \mathbb{P}^n$ of degree $k$ gives rise to a rational map $\mathbb{P}^n \dasharrow \text{Sym}^k X$ which is birational on its image. This map associates to a generic $x\in X$ the fiber $f^{-1}(x)\in \text{Sym}^k(X)$. The Zariski closure $\Gamma$ of the image of this rational map is contained in an \emph{orbit of degree $k$ for rational equivalence}, namely a fiber of $\text{Sym}^kX\to CH_0(X)$ or $X^k\to CH_0(X)$. In the notation of \cite{M} $\Gamma$ is an $n$-dimensional \emph{constant-cycle subvariety of degree} $k$ ($\text{CCS}_k$), namely an $n$-dimensional subvariety of $\text{Sym}^kX$ or $X^k$ contained in a fiber of $\text{Sym}^kX\to CH_0(X)$ or $X^k\to CH_0(X)$ respectively.

Similarly, given a dominant rational map $f: Y\dasharrow X$ from an $n$-dimensional variety $Y$ with $\text{irr}(Y)=k$, we get an $n$-dimensional subvariety of an orbit for rational equivalence of degree $k$ of $X$. Indeed, a degree $k$ rational map $Y\dashrightarrow \mathbb{P}^n$ gives rise to a rational map $\mathbb{P}^n\dasharrow \text{Sym}^k Y$. Let $\Gamma$ be the closure of the image of this map and $\widehat{\Gamma}$ be its preimage in $Y^k$. The projection onto the first factor $\widehat{\Gamma}\to Y$ is dominant. The image of $\widehat{\Gamma}$ under $f\times f\times \ldots\times f : Y^k\dasharrow X^k$ is contained in an orbit for rational equivalence and its projection onto the first factor is dominant. It follows that this image is an $n$-dimensional subvariety of an orbit for rational equivalence.
\end{remark}

The upshot is that, given $X$ such that $X^k\to CH_0(X)$ has no $n$-dimensional fibers, any rational map from $X$ to $\mathbb{P}^n$ has degree at least $k+1$, i.e. $\text{irr}(X)>k$. Similarly, any $n$-dimensional $Y$ that admits a dominant rational map to $X$ cannot admit a dominant rational map of degree less than $k+1$ to $\mathbb{P}^n$. This shows that in such a situation the \emph{degree of uni-irrationality} $\text{uni.irr}(X)$ of $X$ is at least $k+1$. Recall that 
$$
\text{uni.irr}(X):= \min \{\text{irr}(Y)\mid \dim Y=n, \exists Y\dashrightarrow X \text{ dominant} \}.
$$

One can also consider the following birational invariants:
\begin{align*}\text{irr}_d(X) &:= \min \{\delta\mid \text{ for } x\in X\text{ generic},\; \exists Y\subset X \text{ with } \dim Y=d,\; x\in Y, \text{ and } \text{irr}(Y)=\delta \},\\
\text{uni.irr}_d(X)&:= \min \{\delta\mid \text{ for } x\in X \text{ generic},\; \exists Y\subset X \text{ with } \dim Y=d,\; x\in Y, \text{ and } \text{uni.irr}(Y)=\delta \}.
\end{align*}
They interpolate between the covering gonality $\text{cov.gon}(X):=\text{irr}_1(X)=\text{uni.irr}_1(X)$ and the degree of irrationality and uni-irrationality respectively. Note that $\text{cov.gon}(X)$ may be thought of as a measure of the failure of $X$ to be uniruled and that $\text{cov.gon}(X)=1$ if and only if $X$ is uniruled.

As a consequence of Remark \ref{remark}, one can give lower bounds on $\text{uni.irr}_d(X)$ (and so a fortiori on $\text{irr}_d(X)$) by giving lower bounds on the following birational invariant, which was introduced by Voisin in the case $d=1$ (see \cite{BDPELU}):
\begin{align*}\nu_d(X):=\text{min}\begin{cases}\delta\in \mathbb{N}\;\Big{\vert}\begin{rcases} \text{for a general }x\in X, \; \exists\, x_2,\ldots,x_\delta\in X\text{ such that the fiber of }\\ X^\delta\to CH_0(X) \text{ containing } (x, x_2,\ldots, x_\delta) \text{ has dim. at least } d\end{rcases}.\end{cases}\end{align*}
In particular, we have
$$1\leq \nu_d(X)\leq \text{uni.irr}_d(X)\leq \text{irr}_d(X).$$

In the case where $X=A$ is an abelian variety, the bound $\text{irr}(A)\geq \dim A + 1$ from \cite{AP2} can be promoted to an inequality $\nu_{\dim A}(A)\geq \dim A+1$ as observed in \cite{M}. Indeed, this follows directly from Th. 1.4 (1) from \cite{V}, which states that the dimension of any orbit in $A^k$ is less than $k$. In \cite{M} Th. 5.2 the inequality $\nu_{\dim A}(A)\geq \dim A + 1$ is improved to $\nu_{\dim A}(A)\geq \dim A + 2$, for $A$ a very general abelian variety of dimension $g\geq 3$.

The bound is obtained by degenerating to an abelian variety
isogenous to the product of an abelian variety of dimension $g-1$ and an elliptic curve, and by projecting the constant cycle subvariety to the power of the $(g-1)$-dimensional factor. This is the first step in a strategy of degeneration and projection introduced in \cite{P}. This method has crucially been used inductively under some guise in \cite{AP}, \cite{V}, and \cite{M}. Our contribution is to show that the induction proposed by \cite{M} can be applied from dimension $g$ to dimension $2$. We can use this to get a contradiction with known results about families of constant cycle subvarieties of degree 1 on hyper-K\"ahler varieties (see \cite{M} Cor. 2.12 or \cite{V2} Th. 1.3 (i)) to obtain:

\begin{TEO}\label{mainthm}
If $A$ is a very general abelian variety of dimension at least $3$ and $d\in\mathbb{Z}_{>0}$ (and $d\leq \dim A$ if $\dim A$ is odd), then 
$$\nu_{d}(A)\geq d+(\dim A+1)/2.$$
In fact $A$ does not admit any $d$-dimensional orbits of degree $k$ for 
$$k< d+(\dim A+1)/2.$$
\end{TEO}

This extends the results of \cite{M} which covered the cases $d=1,2$. As a geometric application we obtain the following new bound:
\begin{COR}
If $A$ is a very general abelian variety of dimension at least $3$ and $0<d\leq \dim A$, then
$$\textup{irr}_d(A)\geq \textup{uni.irr}_d(A)\geq d+ (\dim A+1)/2.$$
In particular, taking $d=\dim A$, we have
$$\textup{irr}(A)\geq \textup{uni.irr}(A)\geq (3\dim A+1)/2.$$
\end{COR}

A very similar argument gives the following theorem and corollary:
\begin{TEO}\label{strong}
If $1\leq d\leq \dim A$ and a very general abelian variety $A$ of dimension at least $4$ contains a one-dimensional family of $d$-dimensional orbits of degree $k$, then 
$$k\geq d+\dim A/2+1$$
\end{TEO}
Note that this theorem follows trivially from Thm. \ref{mainthm} in the case where $\dim A$ is even. We can use this last theorem to rule out the existence of some families of correspondences. We will consider $d$-dimensional subvarieties $Y$ of $A\times \mathbb{P}^d$ with generically finite projection to $A$ and $\mathbb{P}^d$, and call the degree $\text{deg}(Y)$ of such a correspondence the degree of the projection $Y\to \mathbb{P}^d$. Moreover, we will assume that the map $\mathbb{P}^d\to \text{Sym}^kA$ arising from looking at the fibers of the projection of $Y$ on the second factor is injective. In particular, the image of such a correspondence $Y$ under the projection to $A$ is a $d$-dimensional subvariety of $A$ with degree of uni-irrationality at most $\text{deg}(Y)$. We roughly think of such a correspondence as a geometrically tractable approximation to a ``degree of uni-irrationality datum'', namely the datum of 
$$(Y\subset X, f: Z\dashrightarrow Y, h: Z\dashrightarrow \mathbb{P}^d),$$
where $f$ and $h$ are dominant rational maps.

In this spirit, the following corollary should be seen as a sort of rigidity statement about such data. Though its range of applications seems limited considering that Theorem 1.3 is likely not sharp, it serves as a prototype of geometric results that can be harnessed from the non-existence of families of high-dimensional constant cycle subvarieties of small degree. Here we use standard scheme-theoretic notation and, for varieties $X/\mathbb{C}$, $C/\mathbb{C}$, we write $X_C$ for the base change of $X\to \text{Spec}(\mathbb{C})$ by $C\to \text{Spec}(\mathbb{C})$, i.e. for the variety $X\times C$.
\begin{COR}\label{rigid}
%A very general abelian variety of dimension $g=2l+1$ does not contain any one-dimensional families of $d$-dimensional subvarieties with degree of (uni)irrationality less than or equal to $d+l+1$.
Let $A$ be a very general abelian variety of dimension $g=2l+1$ and $C$ be a curve. Consider $\mathcal{Y}\subset (A\times \mathbb{P}^d)_C$ with relative dimension $d$ over $C$ and $1\leq d\leq \dim A$. Suppose that the projections to $A_C$ and $\mathbb{P}^d_C$ are both generically finite and that the morphism $\mathbb{P}^d_C\to (\textup{Sym}^{\deg(\mathcal{Y})}A)_C$ arising from looking at the fibers of the second projection is injective. Then the degree of the projection of $\mathcal{Y}$ to $A_C$ is at least $d+l+1$. Moreover, if this degree is $d+l+1$, there is a morphism $\phi: C\to \textup{Aut}(\mathbb{P}^d)$ and a $Y_0\subset A\times \mathbb{P}^d$ such that $\mathcal{Y}=\{[a, \phi(c)(x),c]: (a,x)\in Y_0\}\subset A\times \mathbb{P}^d\times C=(A\times\mathbb{P}^d)_C$.
\end{COR}

In the second section, we prove the results above contingent on Claim \ref{claim} using methods from \cite{M}. In the last section, we use Generic Vanishing theory to prove a projection lemma and use it in turn to prove Claim \ref{claim}.

\section{Proofs}

\begin{proof}[Proof of Theorem \ref{mainthm}] Let $ \mathcal{A}\to S$ be a locally complete family of abelian varieties of dimension $g$, that is, a family such that the natural map from $S$ to the corresponding moduli space of abelian varieties is dominant and
generically finite.

Suppose that a very general abelian variety in this family has a $d$-dimensional orbit for rational equivalence of degree $k$. Basic facts about rational equivalence imply that, after a convenient base change, we have a family $\Gamma\subset \mathcal{A}^k$ of $d$-dimensional subvarieties such that $\Gamma_s$ is contained in a fiber of $\mathcal{A}_s^k\to CH_0(\mathcal{A}_s)$, for every $s\in S$. Here and in the rest of this paper we write $\mathcal{A}^k$ for $\mathcal{A}_S^k:=\mathcal{A}\times_S\mathcal{A}\times_S\ldots\times_S\mathcal{A}$ to simplify notation.

Let $S_{\lambda }\subset S$ be the loci along which $\mathcal{A}_s$ is isogenous to $\mathcal{B}_s^\lambda\times \mathcal{E}_s^\lambda$, where $\mathcal{B}^\lambda/S_\lambda$ and $\mathcal{E}^\lambda/S_\lambda$ are families of abelian $(g-1)$-folds and elliptic curves respectively. Here $\lambda $ encodes the type of the isogeny and belongs to an indexing set $\Lambda_{g-1}$. Similarly, let $S_\tau\subset S$ be the loci along which $\mathcal{A}_s$ is isogenous to $\mathcal{D}_s^\tau\times \mathcal{F}_s^\tau$, where $\mathcal{D}^\tau/S_\tau$ and $\mathcal{F}^\tau/S_\tau$ are families of abelian surfaces and $(g-2)$-folds respectively. Here $\tau$ encodes the type of the isogeny and belongs to an indexing set $\Lambda_{2}$. Denote by $\mathcal{A}'/S_\lambda'$ (resp. $\mathcal{D}'/S_\tau'$) the universal family of abelian $(g-1)$-folds (resp. surfaces) to which the family $\mathcal{B}^\lambda/S_\lambda$ (resp. $\mathcal{D}^\tau/S_\tau$) maps, and by $\phi_\lambda: (\mathcal{B}^\lambda)^k\to \mathcal{A}'^k/S_\lambda'$ (resp. $\phi_\tau: (\mathcal{D}^\tau)^k\to \mathcal{D}'^k/S_\tau'$) the $k$-fold fiber product of the  corresponding morphism. Finally, call $p_{\lambda}$ (resp. $p_\tau$) the projection $\mathcal{A}_{S_\lambda}^k\to (\mathcal{B}^\lambda)^k$ (resp. $\mathcal{A}^k_{S_\tau}\to (\mathcal{D}^\tau)^k$), which exists after passing to a generically finite cover of the base.\\

Following \cite{M} we define condition (*) (for $l=2$) as follows : 
\vskip 3mm
$(*)$ The subset $R_{gf}=\bigcup_{\tau} \{s\in S_{\tau} \mid \dim p_\tau(\Gamma_s)=\dim \Gamma_s \}\subset S$ is dense.%\footnote{Equivalently it is non-empty, see proof of Claim \ref{claim}}.
\vskip 3mm

The proof of the theorem rests crucially on the following claim:

\begin{CLA}\label{claim}
A flat family $\Gamma\subset \mathcal{A}^k$ of constant cycle subvarieties (of relative dimension $d$ over $S$) satisfies condition $(*)$ for $l=2$ provided that, either $g=\dim_S \mathcal{A}$ is even, or $1\leq d\leq g$ and $k\leq 2g-2$.
\end{CLA}
We complete the proof of Theorem \ref{mainthm} assuming this claim which will be proved at the end of this note.
Since any $\Gamma_s$ is contained in an orbit, by \cite{M} Lem. 3.4, $\Gamma$ also satisfies :
\vskip 3mm
$(**)$ The subset $\bigcup_{\tau}\{s\in R_{gf}\cap S_\lambda \mid p_\tau(\Gamma_s) \text{ is not covered by tori}\}\subset R_{gf}$ is dense.
\vskip 3mm
Therefore it is possible to apply \cite{M} Prop. 3.5 and Prop. 4.1 in order to see that for some $\tau\in \Lambda_2$, the subvariety $\phi_\tau(p_\tau(\Gamma_{S_\lambda}))\subset {\mathcal{D}'}^{k,\, 0}:=\ker({\mathcal{D}'}^k\to \mathcal{D}')$ has dimension at least $d+g-2$. Roughly speaking, the basic idea of the induction is the following: one can show that for suitable $\lambda\in \Lambda_{g-1}$ the relative dimension of $\phi_\lambda(p_\lambda(\Gamma_{S_\lambda}))/S_\lambda'$ is $d+1$. This amounts to saying that, when projecting to $B_s^k$ and varying the elliptic factor $E_s$, the union of the image has dimension $d+1$. The argument uses the fact that $p_{\lambda,s}|_{\Gamma_s}$ is generically finite on its image for a dense set of $s\in \bigcup S_\lambda\subset S$, which is a consequence of $(*)$. If we can ensure that $\phi_\lambda(p_\lambda(\Gamma))/S_\lambda'$ still satisfies condition $(*)$, and this is the key technical step, then we can proceed with the induction since we will have the desired generic finiteness at each step.

To achieve this, one considers a further specialization to $\mathcal{A}_s$  isogenous to $D_s\times F_s\times E_s$, with $E_s$ an elliptic curve, compatible with a specialization of the generic abelian variety $B_s$ of dimension $g-1$ to one isogenous to $D_s\times F_s$. Then, moving the elliptic curve $E_s$ in moduli, we can project $\Gamma_s\subset \mathcal{A}_s^k$ to either $D_s^k$ or $(D_s\times F_s)^k$. In both cases the image will vary as we vary the elliptic factor and thus we will get a $(d+1)$-fold $Z_1\subset D_s^k$ and a $(d+1)$-fold $Z_2\subset (D_s\times F_s)^k$. Since these $(d+1)$-folds are irreducible and $Z_1$ dominates $Z_2$ under the projection $(D_s\times F_s)^k\to D_s^k$, the restriction of this projection to $Z_1$ must be generically finite on its image. This is how condition $(*)$ is verified to hold at each step of the induction. The reader should consult section 4 of \cite{M} for a thorough explanation of this argument.

We can thus find a $\tau\in \Lambda_2$ such that $\phi_\tau(p_\tau(\Gamma_{S_\tau}))\subset {\mathcal{D}'}^{k,\, 0}:=\ker({\mathcal{D}'}^k\to \mathcal{D}')$ has relative dimension $d+g-2$ over $S_\tau'$. Moreover, this variety is foliated by $d$-dimensional constant cycle subvarieties. We can also choose $\tau$ such that, for a generic $D$ in the family $\mathcal{D}^\tau$, the map from $S_\tau(D):=\{s\in S_\tau: \mathcal{D}^\tau_s=D\}$ to an appropriate Chow variety given by $s\mapsto [\phi_\tau(p_\tau(\Gamma_{s}))]$ is generically finite. As long as $g\geq 4$, we can now apply the same argument as in \cite{M} Th. 4.4 to see that $\phi_\tau(p_\tau(\Gamma_{S_\tau}))$ must be foliated by constant cycle subvarieties of dimension at least $d+1$. 

Now observe that by a refinement of the Mumford-Ro\u{\i}tman bound for the size of rational orbits on surfaces (see \cite{M} Cor. 2.12) we have $d+g-2\leq 2k-2-(d+1)$ if $g\geq 4$ and $d+1\leq 2k-2-d$ if $g=3$.
\end{proof}

\begin{proof}[Proof of Theorem \ref{strong}]
It suffices to show that a very general abelian variety of dimension $g=2l+1$ does not admit a one-dimensional family of $d$-dimensional orbits of degree $k:=d+l+1$. We follow the same proof as above, with $\Gamma\subset \mathcal{A}^k$ a family of $(d+1)$-dimensional varieties such that,  for every $s\in S$, the variety $\Gamma_s\subset \mathcal{A}^k_s$ is foliated by constant cycle subvarieties of degree $k$. The key point is that $\Gamma_s$ is totally isotropic for 
$$\omega_k:=\sum\text{pr}_i^*\omega\in H^0(\mathcal{A}_s^k,\Omega_{\mathcal{A}_s^k}^2),$$
 for any $\omega\in H^0(\mathcal{A}_s,\Omega_{\mathcal{A}_s}^2)$. Of course $p_\tau(\Gamma_s)$ is either a degree $k$ constant cycle subvariety or it is foliated in codimension 1 by constant cycle subvarieties. Hence it enjoys the analogous total isotropy property. This was the key property we used about constant cycle subvarieties.

Just as in the proof above, there is a $\tau$ such that $\phi_\tau(p_\tau(\Gamma_{S_\tau}))$ is $(d+g-1)$-dimensional. In fact the method of \cite{M} actually show that there is a non-empty set $\Lambda$ indexing $\tau$ such that $\phi_\tau(p_\tau(\Gamma'_{S_\tau}))$ is $(d+g-1)$-dimensional, for any complement $\Gamma'_{S_\tau}$ of a countable union of Zariski closed subsets of $\Gamma_{S_\tau}$. Moreover, it is easy to see that, for any $\tau\in \Lambda$ and an integrable foliation $\{\mathcal{Z}_t\}_{t\in \Theta}\subset \Gamma_{S_\tau}$ by subvarieties of relative dimension $d$ over $S_\tau$ indexed by $\Theta$, the variety $\phi_\tau(p_\tau(\mathcal{Z}_{t}))$ has relative dimension $d+g-2$ over $S_\tau$ for generic $t\in \Theta$. Moreover, from the way this set $\Lambda$ is obtained it is clear that $\bigcup_{\tau\in \Lambda}S_\tau\subset S$ is dense. We contend that $\tau$ can be chosen so that $\phi_\tau(p_\tau(\Gamma_{S_\tau}))$ is foliated by constant cycle subvarieties of dimension at least $d+1$. It is clear from the way such $\tau$ are obtained in \cite{M} that we can choose $\tau$ such that $p_{\tau}|_{\Gamma_s}$ is generically finite on its image for a generic $s\in S_\tau$. This condition determines a subset $\Lambda'\subset \Lambda$ and we can make sure that $\bigcup_{\tau\in \Lambda'}S_\tau$ is dense in $S$.

Since $\Gamma$ is foliated by $d$-dimensional constant cycle subvarieties, we can find a rational map $f: \Gamma/S\dasharrow \mathcal{C}/S$, where $\mathcal{C}/S$ is a family of curves and the fibers of this map are constant cycle subvarieties of degree $k$. For every $\tau$ such that $f$ is defined on $\Gamma_{S_\tau}$, there is an open subset $\mathcal{C}_{\tau}\subset \mathcal{C}_{S_\tau}$ such that $p_{\tau}|_{\Gamma_c}$ is generically finite on its image for $c\in \mathcal{C}_\tau$. Consider a pencil of hypersurface sections of $\mathcal{C}$. These hypersurface sections dominate the base and can be thought of as multisections $\{s_t\}_{t\in \mathbb{P}^1}$ of $\mathcal{C}\to S$. Now, for a very general $t\in \mathbb{P}^1$, the multisection $s_t$ is contained in $\mathcal{C}_{\tau,s}$ for every $\tau\in \Lambda'$ and generic $s\in S_\tau$. Upon passing to a generically finite cover of $S$, these multisections provide families $\mathcal{Z}_t$ of $d$-dimensional constant cycle subvarieties, for $t\in \mathbb{P}^1$ generic. We can then apply the argument of \cite{M} section 3 to see that there is a $\tau\in \Lambda'$ such that the map $s\mapsto [\phi_\tau(p_\tau({\mathcal{Z}_{t,s}}))]$ from $S_\tau$ to an appropriate Chow variety is generically finite for very general $t\in \mathbb{P}^1$.

We claim that for such a $\tau$ the variety $\phi_\tau(p_\tau(\Gamma_{S_\tau}))$ is foliated by constant cycle subvarieties of degree $k$ of dimension at least $d+1$. Indeed, note that for $t$ very general the variety $\phi_\tau(p_\tau({\mathcal{Z}_{t,S_\tau}}))$ is at least $(d+g-2)$-dimensional and either foliated by constant cycle subvarieties of dimension at least $d+1$ or it is a subvariety of dimension at least $d+g-2+6=d+g+4$, where the number $6$ is the dimension of the moduli of abelian $3$-folds with some fixed polarization. The second case cannot happen by Cor. 2.12 of \cite{M}. Hence, for very general $t$, the variety $\phi_\tau(p_\tau({\mathcal{Z}_{t,S_\tau}}))$ is foliated by $(d+1)$-dimensional orbits. Taking the union of $\phi_\tau(p_\tau({\mathcal{Z}_{t,S_\tau}}))$ over all such $t$, we see that $\phi_\tau(p_\tau(\Gamma_{S_\tau}))$ is itself foliated by $(d+1)$-dimensional orbits. This provides a contradiction to Cor. 2.12 of \cite{M}, the refinement of the Mumford-Ro\u{\i}tman bound for the size of rational orbits on surfaces.
\end{proof}

\begin{proof}[Proof of Corollary \ref{rigid}]
Let $\text{pr}_1: (A\times \mathbb{P}^d)_C\to A_C$ and $\text{pr}_2: (A\times \mathbb{P}^d)_C\to \mathbb{P}^d_C$ be the projections. Considering fibers of $\text{pr}_2|_{\mathcal{Y}}$ gives a morphism $\eta: \mathbb{P}^d_C\to (\text{Sym}^kA)_C$ whose image is a family of constant cycle subvarieties, namely $\eta(\mathbb{P}^d_c)\subset \text{Sym}^kA$ is a constant-cycle subvariety for every $c\in C$. If ${\textup{pr}_1}|_{\mathcal{Y}}$ has degree less than $d+l+2$ then this family must be constant. Consider the graph $\Gamma\subset \mathbb{P}^d_C\times_C(\text{Sym}^kA)_C=\mathbb{P}^d\times\text{Sym}^kA\times C$ of $\eta$ and the fiber product:
$$R:=\Gamma\times_{\text{Sym}^kA}\Gamma=\{(t,t')\in \mathbb{P}^d_C\times \mathbb{P}^d_C: (\text{pr}_2|_{\mathcal{Y}})^{-1}(t)=(\text{pr}_2|_{\mathcal{Y}})^{-1}(t')\in\text{Sym}^kA\}\subset \mathbb{P}^d_C\times \mathbb{P}^d_C.$$
Fixing $0\in C$, we see that the projection of $R_0\subset \mathbb{P}^d\times \mathbb{P}^d_C$ onto the second factor is an isomorphism. Moreover, for every $c\in C$, the projection of the fiber $R_{0,c}\subset \mathbb{P}^d\times \mathbb{P}^d$ onto the first factor is an isomorphism. As a consequence, we can think of $R_0$ as a family of automorphisms of $\mathbb{P}^d$ indexed by $C$. We thus get a morphism $\phi$ from $C$ to the automorphism group of $\mathbb{P}^d$ and $\mathcal{Y}=\{[a,\phi(c)(x),c]: (a,x)\in Y_0\}\subset A\times\mathbb{P}^d\times C$.
\end{proof}

\vskip 3mm

\section{Generic Vanishing and Proof of Claim \ref{claim}}

Given an abelian subvariety $T$ of an abelian variety $A$, let $p_T: A\to A/T$ be the quotient map. Denote by $\text{Sub}(A)$ the poset of positive-dimensional abelian subvarieties of $A$ under inclusion. The proof of Claim \ref{claim} uses the following lemma:

\begin{LEM}\label{fin} Let $X$ be a subvariety of an abelian variety $A$. There is a finite subset $S_X\subset \textup{Sub}(A)$ such that if $T\in \textup{Sub}(A)$ is such that $p_{T|X}$ is not generically finite on its image and $p_T(X)$ is not covered by tori, then $T$ contains an element of $S_X$.
\end{LEM}
Note that under the above assumptions $p_T(X)$ is positive-dimensional, else it would be covered by tori.

\begin{proof}[Proof of Lem. \ref{fin}] The result is a well known consequence of Generic Vanishing theory. We include it for completeness. We easily reduce to the case
$A=\text{Alb}(\tilde X)$, where $\tilde X$ smooth, and $X$ is the image of $\tilde X$ under the albanese map $\text{alb}_{\tilde X}: \tilde X\to A$, which is birational on its image. Let $T$ be a subtorus of $A$ such that $p_{T|X}$  is not generically finite on its image and $p_T(X)$ is not covered by tori. Changing the desingularization $\tilde X\to X$ if needed, denote by $Y\to p_T(X)$ a desingularization of $p_T(X)$ such that the rational map $\tilde X\dasharrow Y$ extends to a morphism $g: \tilde X\to Y$. The quotient $A\to A/T$ factorizes via $\text{Alb}(Y)$. Note that $ g^*\text{Pic}^0( Y)\subset S^{\dim  Y}(\tilde X)$, where
$$S^k(\tilde X):=\{\alpha\in \text{Pic}^0(\tilde X)|h^k(\tilde X,\alpha)>0\}$$ are the cohomological support loci (see \cite{GL}). To see this, observe that by hypothesis $p_T(X)$ is not covered by tori, hence the same holds for the albanese image of $Y$. In particular, $\chi(K_{Y})>0$ by Th. 3 in \cite{EL}. Therefore, by generic vanishing, for any $\beta\in \text{Pic}^0(Y)$ we have $h^{\dim Y}(Y,\beta)>0$ and thus $h^{\dim Y}(\tilde X, g^*\beta)>0$. Hence  $ g^*\text{Pic}^0(Y)$ is contained in an irreducible component $W$ of some $S^k(\tilde X)$, with  ${k<\dim X}$.  Since all the irreducible components of $S^k(\tilde X)$ are translates of subtori, $W$ is an abelian variety. Dualizing we get the factorization 
$$p_T:A\rightarrow W^*\rightarrow \text{Alb}(Y)\rightarrow A/T.$$The lemma then follows by the observation that the number of irreducible components $W$ of $\bigcup_{k<\dim X}S^k(\tilde X)$ is finite.
\end{proof}

Let $B$ be an abelian variety and $r,l,k$ be positive integers with $r\geq 2$ and $1\leq l\leq r-1$. Given $M\in M_{r\times l}(\mathbb{Z})$ we denote by $i_M: B^l\to B^r$ the morphism given by $(b_1,\ldots, b_l)\mapsto M(b_1,\ldots, b_l)^t$ and by $p_M: [B^r]^k\to [B^r/i_M(B^l)]^k$ the quotient map. We use the previous lemma to deduce the following:

\begin{PLEM}\label{projlem}
Let $X\subset [B^r]^k$ be a subvariety which is not covered by tori. A generic $M\in M_{r,l}(\mathbb{Z})$ is such that $p_{M}(X)$ is covered by tori or $p_{M|X}: X\to [B^r/i_M(B^l)]^k$ is generically finite on its image.
\end{PLEM}
\begin{proof}
By Lemma \ref{fin} there is a finite number of abelian subvarieties $F_a\subset [B^r]^k$ such that if $p_{M}(X)$ is not covered by tori and $p_{M|X}$ is not generically finite on its image then $[i_M(B)]^k$ contains some $F_a$. It is then enough to show that for the generic $M\in M_{r,l}(\Z)$, the abelian subvariety $[i_M(B^{l})]^k\subset [B^r]^k$ does not contain any such subvariety $F_a$. This is almost obvious: $T_0B\cong H_1(B,\mathbb{Z})\otimes \mathbb{R}$ and a vector $\mathbf{w}^a\in T_0F_a$ is identified with $(\underline{w}^a_1,\ldots,\underline{w}^a_k)$
where $\underline{w}^a_i\in H_1(B,\mathbb{Z})^r\otimes \mathbb{R}$. For each $a$, choose a non-zero vector $\mathbf{w}^a$ such that its components are in $H_1(B,\mathbb{Z})^r\otimes \mathbb{Z}$ and a non-zero component $\underline{w}_{i_a}^a\in H_1(B,\mathbb{Z})^r$ of $\mathbf{w}^a$. Choosing an isomorphism $H_1(B,\mathbb{Z})\cong \mathbb{Z}^{2\dim B}$, we get an isomorphism $H_1(B,\mathbb{Z})^r\cong [\mathbb{Z}^r]^{2\dim B}$ and, again, for each $a$, choose a non-zero component $\underline{u}_a\in \mathbb{Z}^r$ of $\underline{w}_{i_a}^a\in H_1(B,\mathbb{Z})^r$. It is enough to take $M\in M_{r,l}(\Z)$ such that $\underline{u}_a$ is not in the space generated by the columns of $M$ for any $a$. The set of such $M$ is Zariski open in $M_{r,l}(\Z)$ and is not all of $M_{r,l}(\Z)$.
\end{proof}

Finally we prove Claim \ref{claim}:

\begin{proof}[Proof of Claim \ref{claim}]

First observe that by Lemma 3.1 of \cite{M}, up to passing to a Zariski open in $S$, we can assume that $\Gamma_s$ is never covered by tori. %Moreover, to show that $(*)$ is satisfied it suffices to show that $R_{gf}$ is non-empty, i.e. $R_{gf}$ is dense if and only if it is non-empty. Indeed, if $p_\tau|_{\Gamma_{s_0}}$ is generically finite on its image then so is $p_\tau|_{\Gamma_{s}}$ for generic $s\in S_\tau$. Consider such an $s\in S_\tau$ corresponding to an abelian variety which is isogenous to a power of an elliptic curve. This $s$ will belong to many $S_{\tau'}$ and an easy argument using the Gauss map shows that $p_{\tau'}|_{\Gamma_s}$ will be generically finite for most $\tau'$. This is enough to get density. Indeed $R_{gf}$ contains a countable collection of locally closed subsets going through $s$ (opens in some of the $S_{\tau'}$). The union of the tangent cones at $s$ to these locally closed subsets is dense in $\mathbb{P}T_sS$.
We first reduce to the case where $\dim A=g$ is even. It is in the course of this reduction that we will use the hypothesis that $1\leq d\leq \dim A$ if $g$ is odd. Suppose that $g$ is odd, we will proceed exactly as in the proof of Th. 5.2 of \cite{M} to show that we can find a locus $S_\lambda(E)\subset S$ along which $\mathcal{A}_s$ is isogenous to $B_s\times E$, where $E$ is a fixed elliptic curve, and such that the restriction of the projection of $\Gamma_s$ to $B_s^k$ is generically finite on its image for generic $s\in S_\lambda(E)$. Then the image of $\Gamma_{S_\lambda(E)}$ under this projection will be a family of constant cycle subvarieties $\Gamma'/S'$ in a locally complete family of abelian $(g-1)$-folds $\mathcal{A}'/S'$. We are reduced in this way to the case of even-dimensional abelian varieties.

Consider $s\in S$ such that $\mathcal{A}_s$ is isogenous to $E^g$ for some elliptic curve $E$. We contend that there is a choice of $M\in M_{g,1}(\mathbb{Z})=\mathbb{Z}^g$ such that the composition 
$$\Gamma_s\to \mathcal{A}_s^k\to [E^g]^k\to [E^g/i_M(E)]^k$$
of the inclusion, the isogeny, and the projection is generically finite on its image. This is obvious if $d<\dim A$ so it suffices to show it for $d=\dim A$.

If this were not the case, it is easy to see (see Lemma 5.4 of \cite{M}) that, for any $y$ in the smooth locus of $\Gamma_s$, the tangent space to $\Gamma_s$ at $y$ would be of the form $M(T_0\mathcal{A}_s)$, for some $M\in \mathbb{C}^k$. Here, given $M=(m_1,\ldots, m_k)\in \Co^k$, we abuse notation and write $M$ for the map $T_0\mathcal{A}_s\to T_0\mathcal{A}_s^k=(T_0\mathcal{A}_s)^k$ given by $v \mapsto (m_1v,\ldots, m_kv)$. But, given $\omega\in H^0(\mathcal{A}_s,\Omega_{\mathcal{A}_s}^q)$, the form $\omega_k:=\sum_{i=1}^k \text{pr}_1^*\omega$ restricts to zero on $\Gamma_s$ since it is a constant cycles subvariety. It follows that the image of the Gauss map of $\Gamma_s$ does not only lie in the subvariety of $g$-dimensional subspaces of $T_0\mathcal{A}_s^k$ of the form $M(T_0\mathcal{A}_s)$ for $M\in \Co^k$, but in the subvariety of such subspaces where $M=(m_1,\ldots, m_k)$ satisfies $\sum_{i=1}^k m_i^q=0$, for $i=1,\ldots, q$. This imposes $q-1$ independent conditions on $M\in\Co^k$ and so the image of the Gauss map of $\Gamma_s$ must be at most $k-(q-1)$-dimensional. As long as $k\leq 2g-2$ this show that the Gauss map of $\Gamma_s$ is not generically finite. Since $\Gamma_s$ is not covered by tori, this contradicts well-known results about non-degeneracy of the Gauss map of subvarieties of abelian varieties (see (4.14) in \cite{GH}).

In the case of abelian varieties of even dimension $g=2n$, consider the locus $Y\subset S$ of abelian varieties isogenous to $D^{n}$, for some simple abelian surface $D$. Since $Y$ is dense in $S$, to show that $(*)$ is satisfied it suffices to show that $Y\subset R_{fg}$. For every $y\in Y$, we can fix an isogeny $\mathcal{A}_y\sim D^{n}$. Then, for any $M\in M_{n,n-1}(\Z)$ of maximal rank, we get a projection $p_{M}:\mathcal{A}_y\to D^k$ obtained by composing the fixed isogeny with the projection map $[D^{n}]^k\to [D^{n}/i_M(D^{n-1})]^k$. The maps $p_M$ are specializations of maps $p_{\tau}$ to $y\in Y$. But for any $y\in Y\cap S_{\tau}$ the variety $p_\tau(\Gamma_y)$ is a constant cycle subvariety, hence is not covered by tori. Indeed, up to an isogeny, $p_\tau(\Gamma_y)$ is a subvariety of $D^k$, and if $p_\tau(\Gamma_y)$ were covered by tori it would be covered by tori of the form $i_M(D)$, for some $M\in \Z^k$. This, together with the fact that $p_\tau(\Gamma_y)$ must be totally isotropic for the $2$-form $\omega_k:=\sum_{i=1}^k \text{pr}_i^*\omega$, for any $\omega\in H^0(D,K_D)$, would again contradict Lemma 3.1 of \cite{M}. We are thus in a position to apply the Projection Lemma \ref{projlem} to see that for any such $y$ there is an $S_\tau$ containing $y$ such that $p_\tau|_{\Gamma_y}$ is generically finite on its image. It follows that $Y\subset R_{gf}$.
\end{proof}

\end{document}